\documentclass[12pt]{article}
\usepackage{amsthm, amsfonts}
\usepackage{url,color}
\usepackage{amscd,amsmath,amssymb,color}

\newtheorem{definition}{Definition}
\newtheorem{theorem}{Theorem}
\newtheorem{lemma}{Lemma}
\newtheorem{proposition}{Proposition}
\newtheorem{corollary}{Corollary}

\def\eps{\varepsilon}

\def\al{\alpha}
\def\Ga{\Gamma}

\def\La{\Lambda}
\def\la{\lambda}

\def\kappa{\varkappa}
\def\si{\sigma}

\def\C{{\mathbb C}}
\def\N{{\mathbb N}}
\def\Y{\mathbb Y}
\def\P{\mathbb P}
\def\YS{\operatorname{YS}}
\def\SY{{\cal S}^{\rm Y}}

\def\S{{\mathfrak S}}
\def\codim{\operatorname{codim}}
\def\sgn{\operatorname{sgn}}
\def\id{\operatorname{id}}
\def\Br{\operatorname{Br}}
\def\Stab{\operatorname{Stab}}

\newcommand{\A}{{\mathcal{A}}}

\def\beq{\begin{equation}}
\def\eeq{\end{equation}}
\def\bea{\begin{eqnarray*}}
\def\eea{\end{eqnarray*}}

\def\Reg{\operatorname{Reg}}
\def\Ind{\operatorname{Ind}}

\def\Poin{\operatorname{Poin}}

\urldef\vershik\url{avershik@pdmi.ras.ru}
\urldef\natalia\url{natalia@pdmi.ras.ru}
\urldef\yuz\url{yuz@uoregon.edu}

\author{N.~V.~Tsilevich\thanks{%
St.~Petersburg Department of Steklov Institute of Mathematics and
St.~Petersburg State University, St.~Petersburg, Russia.
E-mail: \natalia. Supported by the RFBR grant 17-01-00433.}
\and A.~M.~Vershik\thanks{%
St.~Petersburg Department of Steklov Institute of Mathematics and
St.~Petersburg State University, St.~Petersburg, Russia; Institute for Information Transmission Problems, Moscow, Russia.t
E-mail: \vershik. Supported by the RFBR grant 17-01-00433.}
\and S.~Yuzvinsky\thanks{University of Oregon, Eugene OR, USA. E-mail: \yuz.}}
\title{The intrinsic hyperplane arrangement in an arbitrary irreducible representation of the symmetric group}

\date{}

\begin{document}
\maketitle

\begin{abstract}
For every irreducible complex representation~$\pi_\lambda$ of the symmetric group~$\S_n$, we construct, in a canonical way, a so-called intrinsic hyperplane arrangement~$\A_{\lambda}$ in the space of~$\pi_\lambda$.
This arrangement is a direct generalization of the classical braid arrangement (which is the special case of our construction corresponding to the natural representation of~$\S_n$), has a natural description in terms of invariant subspaces of Young subgroups, and enjoys a number of remarkable properties.
\end{abstract}

\section{Introduction}

In this paper, for an arbitrary irreducible complex representation~$\pi_\la$ of the symmetric group~$\S_n$ we construct an arrangement of hyperplanes  in the space of~$\pi_\la$. This arrangement, which we have called ``intrinsic,'' is defined canonically in representation-theoretic terms. In the case of the natural representation of~$\S_n$, it coincides with the so-called braid arrangement, studied, in particular, by Arnold~\cite{Arnold} in connection with the cohomology of the group of colored braids. An attempt to generalize Arnold's construction to other irreducible representations of symmetric groups has led us to quite dissimilar arrangements, whose complements, in particular, are not~$K(\pi,1)$ spaces.

Let us recall the main result of~\cite{Arnold}.
Consider the natural representation of the symmetric group $\S_n$  by permutations of coordinates in~$\C^n$.
For any distinct $i,j\in\{1,\ldots,n\}$, let $H_{ij}=\{z\in V: z_i=z_j\}$ be the set of fixed points of
the transposition $(ij)\in\S_n$; obviously, $H_{ij}$ is a hyperplane in~$V$ and $(ij)$ acts as the reflection with respect to this hyperplane (``mirror''). The collection of the ${n\choose 2}$ mirrors~$H_{ij}$ is called the {\it braid arrangement}~$\Br_n$ of hyperplanes. Let $M$ be the complement to all these mirrors. Arnold~\cite{Arnold} proved that  $M$ is a $K(\pi,1)$ space with $\pi$ being the group of  colored braids, whence
the cohomology ring~$H^*(M)$ of~$M$ is isomorphic to the cohomology ring of the group of colored braids. He also proved that the Poincar\'e polynomial of~$M$ is equal to
$\Poin(M,t)=(1+t)(1+2t)\ldots(1+(n-1)t)$.

Arnold's results aroused much interest and were generalized in several directions. For instance, Brieskorn~\cite{Brieskorn}
proved that the ring $H^*(M)$ for an arbitrary
arrangement of hyperplanes
is isomorphic to the ring generated by all the  logarithmic differential forms
$\frac {d\alpha}{\alpha}$
where $\alpha=0$ is the equation of a~hyperplane. Then Orlik and Solomon~\cite{OS} showed that this ring
is determined just by the intersection lattice of the arrangement (see~\cite{OT}).
There are also results for the case where  the permutation group~$\S_n$ acts on~$M$, and
hence on~$H^*(M)$. It was proved in~\cite{Lehrer} that for the braid arrangement in~$\C^n$
and the natural action of
$\S_n$, the module~$M$  is isomorphic
 to $2\Ind_{\S_2}^{\S_n}\id_2$, where $\id_2$ is the identity representation of~$\S_2$. It is important to note that the braid arrangement and its direct generalizations to other series of Coxeter groups have another description: their hyperplanes are the mirrors (sets of fixed points) of elements of some finite reflection groups. The arrangements introduced in this paper do not, in general, have this property.

We suggest a development of the described framework in quite another direction, replacing the simplest natural representation of the symmetric group  by other irreducible representations. Here, the naive approach, obviously, fails, since the set of fixed points of a transposition for a general irreducible representation is no longer a hyperplane. However, an analysis from the point of view of the representation theory of~$\S_n$ suggests the correct approach.

The main result of the paper says that {\it for every irreducible complex representation~$\pi_\lambda$ of the symmetric group~$\S_n$ there exists a~canonical ``intrinsic'' hyperplane arrangement~$\A_{\lambda}$ in the space~$V_\la$ of this representation.}
This arrangement is a direct generalization of the braid arrangement, has a natural description in terms of invariant subspaces of Young subgroups, and enjoys a number of remarkable properties.

There are several constructions of this intrinsic arrangement. For instance, one can consider the space of the representation of~$\S_n$ induced from some ``generalized Young subgroup'' (a product of wreath products of symmetric groups associated in a natural way with~$\la$) containing~$\pi_\la$. Then $\A_{\lambda}$
is the intersection of the coordinate (Boolean) arrangement in the space of this induced representation with the subspace corresponding to~$\pi_\la$  (see Theorem~\ref{th:tensor}).

An important property of the intrinsic arrangement~$\A_{\lambda}$ is that the collection of unit normal vectors to its hyperplanes is an orbit of an action of the symmetric group; the convex hull of this orbit is an important polytope.

The general theory of hyperplane arrangements is the subject of much research (see~\cite{OT}).
From the point of view of this theory and the theory of lattices, our examples are, as far as we know, new.\footnote{When this article was under preparation, we have come upon
 the paper~\cite{WWZ} in which closely related objects are considered.
 However, the approach in~\cite{WWZ} is completely different,
in that, first, no hyperplane arrangements are considered, and, second, the definitions are given just in a purely combinatorial form, while our constructions are systematically defined in invariant representation-theoretic terms.} Their main feature and the key point of our approach is that we consider arrangements and lattices invariant under an action of~$\S_n$ and use this invariance and the representation theory of~$\S_n$ to analyze them.  Although there are some papers related to group actions on hyperplane arrangements and lattices, they deal with quite different questions and within a~quite different approach.
Chapter~6 of~\cite{OT} is devoted to arrangements consisting of the fixed hyperplanes of reflections of finite reflection groups;  while
our arrangements are not generated by finite reflection groups, i.e., the groups generated by the  reflections at the hyperplanes are infinite (at least for nontrivial hook diagrams, see
Sec.~\ref{sec:hook}). This poses an independent problem of studying such groups.

To describe the intersection lattice of~$\A_\la$ for an arbitrary irreducible representation~$\pi_\la$ of~$\S_n$ is a quite difficult problem. We consider in detail the simplest nontrivial case, corresponding to hook diagrams of the form~$\la=(n-k,1^k)$.  Here we can already observe important differences with the Arnold case $k=1$. In particular, the complement to~$\A_\la$ in the case $k>1$ is no longer~$K(\pi,1)$, and its fundamental group is Abelian (see Theorem~\ref{th:hookcompl}).

\smallskip

For a background on hyperplane arrangements and related objects, we refer the reader to~\cite{OT}, for that on the classical representation theory of symmetric groups, to~\cite{JK}, and for that on symmetric functions, to~\cite{Mac}.

\smallskip
The paper is organized as follows. In Sec.~\ref{sec:first}, we present our main construction of a hyperplane arrangement~$\A_\la$ canonically associated with an irreducible representation~$\pi_\la$ of~$\S_n$. Section~\ref{sec:tensor} describes an alternative construction of this arrangement as the intersection of the coordinate arrangement in a larger space with the subspace corresponding to~$\pi_\la$. In Sec.~\ref{sec:hook}, we study in detail the hyperplane arrangements corresponding to hook diagrams. Finally, in Sec.~\ref{sec:partlat} we study a natural join homomorphism from the partition lattice~$\Pi_n$ to the lattice of invariant subspaces of Young subgroups, which, in particular, provides a natural characterization of the intrinsic arrangement~$\A_\la$.

\section{The intrinsic hyperplane arrangement associated with an irreducible representation of the symmetric group}\label{sec:first}

First, we introduce necessary notation.

For $n\in\N$, let $\Y_n$ be the set of partitions of~$n$ (or, which is the same, the set of Young diagrams with $n$ cells) and $\P_n$ be the set of partitions of the set $[n]=\{1,\ldots,n\}$. For $\mu\in\Y_n$, we say that a partition $\al\in\P_n$ is of type~$\mu$ if the sizes of the blocks of~$\al$ form the partition~$\mu$ of~$n$; we denote by $\P_n(\mu)$ the set of all partitions $\al\in\P_n$ of type~$\mu$.

For $\la\in\Y_n$, let $\pi_\la$ be the irreducible representation of~$\S_n$ with diagram~$\la$ and $V_\la$ be the space of this representation.

For a set $A$, denote by $\S[A]$ the group of permutations of its elements. Given a partition $\al\in\P_n$ with blocks $A_1,\ldots,A_k$, let $\S_\al=\S[A_1]\times\ldots\times \S[A_k]$ be  the corresponding Young subgroup in $\S_n$. If $\al$ is of type~$\mu$, we say that $\S_\al$ is of type~$\mu$.

Given a set $T$ of transpositions in~$\S_n$, consider the graph~$\Ga(T)$ on the set of vertices $[n]$ in which vertices $i$ and $j$ are connected by an edge if and only if the transposition $(ij)$ lies in~$T$. The partition of~$\Ga(T)$ into connected components determines a partition~$\al_T\in\P_n$, which in turn determines a~Young subgroup~$\S_T:=\S_{\al_T}$.

For an arbitrary element $\si\in\S_n$ or a subgroup $G\subset\S_n$, we denote
$$
V_\la^\si=\{v\in V_\la:\pi_\la(\si) v=v\},\qquad V_\la^G=\{v\in V_\la:\pi_\la(g) v=v\text{ for all }g\in G\}.
$$

The following observation is easy to prove.

\begin{lemma}
Given a collection~$T$ of transpositions in~$\S_n$,  let $V_\la^T=\cap_{\tau\in T}V^\tau_\la$. Then
 ${V_\la^T=V_\la^{\S_T}}$.
\end{lemma}

\begin{lemma}\label{l:kostka}
Let $\S_\al$ be a Young subgroup in~$\S_n$ of type~$\mu$. Then
$$
\dim V_\la^{\S_\al}=K_{\la\mu},
$$
where $K_{\la\mu}$ is the Kostka number.
\end{lemma}
\begin{proof}
The dimension $\dim V_\la^{\S_\al}$ is the multiplicity of the identity representation~$\operatorname{id}_{\S_\al}$ in the restriction of~$\pi_\la$ to~$\S_\al$:
$$
\dim V_\la^{\S_\al}=\langle\operatorname{Res}^{\S_n}_{\S_\al}\pi_\la,\operatorname{id}_{\S_\al}\rangle=\langle\pi_\la,\operatorname{Ind}^{\S_n}_{\S_\al}\operatorname{id}_{\S_\al}\rangle,
$$
where the last equality  is Frobenius reciprocity. Now, using the theory of symmetric functions and observing that the Frobenius map sends $\operatorname{Ind}^{\S_n}_{\S_\al}\operatorname{id}_{\S_\al}$ to the complete symmetric function~$h_\mu$, we obtain
$$
\dim V_\la^{\S_\al}=\langle s_\la,h_\mu\rangle=K_{\la\mu}
$$
(where $s_\la$ is a Schur symmetric function).
\end{proof}

Note that in the case of the standard representation, i.e., $\la=(n-1,1)$, for a single transposition~$\tau\in\S_n$ we have $\dim V_\la^\tau=n-2=\dim V_\la-1$, that is, the invariant subspace of~$\tau$ is a hyperplane, the collection of all such hyperplanes being exactly the braid arrangement $\Br_n$. The correct way of extending this construction to an arbitrary irreducible representation is suggested by the following lemma.

Using the notation and terminology of the theory of lattices, for subspaces~$V_1$ and~$V_2$ of~$V_\la$ we denote by $V_1\vee V_2$ and $V_1\wedge V_2$ their join and meet in the sense of the lattice of subspaces, i.e., $V_1+V_2$ and $V_1\cap V_2$, respectively.

\begin{lemma}\label{l:hyper}
Let $\la\in\Y_n$ and $\la'$ be the Young diagram conjugate to~$\la$. Given a partition $\al\in\P_n(\la')$, consider
an arbitrary collection~$T$ of transpositions in~$\S_n$ generating the Young subgroup~$\S_\al$. Let $H_T=\bigvee_{\tau\in T}V^\tau_\la$. Then $H_T$ depends only on~$\al$, and
$\codim H_T=1$.
\end{lemma}

\begin{proof}
We have
$H_T^\perp=\bigwedge\limits_{\tau\in T}(V^\tau_\la)^\perp$. Now, $(V^\tau_\la)^\perp$ is the eigenspace of the action of~$\tau$ in $V_\la$ with eigenvalue~$-1$, hence $\bigwedge\limits_{\tau\in T}(V^\tau_\la)^\perp$ is the subspace of the sign representation of the Young group~$\S_T=\S_\al$ generated by~$T$. Therefore, the codimension in question is equal to the multiplicity of the sign representation~$\sgn_{\S_T}$ of~$\S_T$ in the restriction of~$\pi_\la$ to~$\S_T$. Thus, similarly to the proof of Lemma~\ref{l:kostka}, we have
\begin{eqnarray*}
\codim H_T&=&\langle\operatorname{Res}^{\S_n}_{\S_T}\pi_\la,\operatorname{sgn}_{\S_T}\rangle
=\langle\pi_\la,\operatorname{Ind}^{\S_n}_{\S_T}\operatorname{sgn}_{\S_T}\rangle\\
&=&\langle s_\la,e_\mu\rangle=\langle s_{\la'},h_\mu\rangle=K_{\la',\mu},
\end{eqnarray*}
where $\mu$ is the type of~$T$. By assumption, $\mu=\la'$, and $K_{\la',\la'}=1$ by the well-known property of Kostka numbers, so the lemma follows.
\end{proof}

Thus, we arrive at the following construction/definition.

\begin{definition}\label{def:main}
{\rm Let $\la\in\Y_n$. For every partition~$\al\in\P_n(\la')$, consider a~collection of transpositions~$T_\al$ generating the Young subgroup~$\S_\al$ and let
$$
H_\al:=\bigvee_{\tau\in T_\al}V^\tau_\la.
$$
By Lemma~\ref{l:hyper}, the subspace~$H_\al$ does not depend on the choice of~$T_\al$ and
is a~hyperplane in~$V_\la$. Thus, the set ${\cal A}_\la:=\{H_\al\}_{\al\in\P_n(\la')}$ is a hyperplane arrangement in~$V_\la$, which will be called the {\it intrinsic arrangement associated with the irreducible representation}~$\pi_\la$ of the symmetric group~$\S_n$.}
\end{definition}

There is a natural transitive action of the symmetric group~$\S_n$ on the set~$\P_n(\la')$.
For ${\al\in\P_n(\la')}$, denote by $n_\al$ a unit normal vector to the hyperplane~$H_\al$ (defined up to a multiplicative constant). Clearly, one can choose the multiplicative constants in such a way that $n_{\si(\al)}=\si(n_\al)$ for every $\al\in\P_n(\la')$ and every $\si\in\S_n$. Thus, {\it the set of normals $\{n_\al\}_{\al\in\P_n(\la')}$ to the hyperplanes of~$\cal A_\la$ is an orbit of~$\S_n$}.

The convex hull of the set of normals  $\{n_\al\}_{\al\in\P_n(\la')}$ is an important polytope (in another form it appeared in~\cite{WWZ} under the name of ``Specht polytope''). For instance, in the case $\la=(n-1,1)$ of the braid arrangement, it is the root polytope of type~$A_n$.

Among the normals $n_\al$ there is a distinguished one in the following sense. Consider the Gelfand--Tsetlin basis $\{e_t\}$ in~$V_\la$, indexed by the standard Young tableaux~$t$ of shape~$\la$ (see, e.g.,~\cite{OV}).
Then there is a unique normal $n_\al$ that coincides (up to a multiplicative constant) with an element of the Gelfand--Tsetlin basis. Namely, let $\la'=(l_1,l_2,\ldots,l_k)$, and let
$$
\al_0=\big\{\{1,\ldots,l_1\},\{l_1+1,\ldots,l_1+l_2\},\ldots,\{l_1+\ldots+l_{k-1}+1,\ldots,n\}\big\}\in\P_n(\la').
$$
Also, consider the ``minimal'' Young tableau $\tau^{\min}_\la$ of shape~$\la$ obtained by successively inserting the numbers $1,2,\ldots,n$ into the first column of~$\la$ from top to bottom, then into the second column of~$\la$ from top to bottom, etc. Then $n_{\al_0}$ coincides with $e_{\tau^{\min}_\la}$ (up to a multiplicative constant), and all the other normals are the elements of the orbit of~$e_{\tau^{\min}_\la}$ under the action of~$\S_n$. The tableau~$\tau^{\min}_\la$ and all the related objects (the partition~$\al_0$, the normal~$n_{\al_0}$, the hyperplane~$H_{\al_0}$, the Young subgroup~$\S_{\al_0}$) will be called {\it distinguished}.

The hyperplanes of our arrangement can also be described as follows. For the symmetric group~$\S[A]$ on a set~$A$ of cardinality~$N$, denote by $Q_A$ the corresponding antisymmetrizer:
$$
Q_A=\frac1{N!}\sum_{\si\in\S[A]}(-1)^{\sgn\si}\si\in\C[\S_n].
$$

\begin{proposition}
For a partition~$\al\in\P_n(\la')$, denote by $R_\al$ the reflection with respect to the hyperplane~$H_\al$. Then
\begin{equation}\label{reflection}
R_\al:=\pi_\la(1-2Q_{A_1}\ldots Q_{A_k})\in\C[\S_n]
\end{equation}
where $A_1,\ldots,A_k$ are the blocks of~$\al$ and $1$ is the identity element of~$\S_n$.
\end{proposition}

\begin{proof}
Consider the distinguished hyperplane~$H_{\al_0}$. Denoting by $P_{\al_0}$ the orthogonal projection to the corresponding normal~$n_{\al_0}$, we obviously have $R_{\al_0}=1-2P_{\al_0}$. But $n_{\al_0}$ coincides with~$e_{\tau^{\min}_\la}$, and it is not difficult to show that the orthogonal projection to~$e_{\tau^{\min}_\la}$  is exactly
$$\pi_\la(Q_{\{1,\ldots,l_1\}}Q_{\{l_1+1,\ldots,l_1+l_2\}}\ldots Q_{\{l_1+\ldots+l_{k-1}+1,\ldots,n\}}),$$
so \eqref{reflection} is proved for~$H_{\al_0}$. Now, since any other hyperplane of~$\A_\la$ is obtained from~$H_{\al_0}$ by the action of an element~$g\in\S_n$, the corresponding reflection is obtained from~$R_{\al_0}$ by the conjugation by~$\pi_\la(g)$, and the proposition follows.
\end{proof}

The number $\#\cal A_\la$ of hyperplanes in the arrangement~$\cal A_\la$ is equal to the number~$\#\P_n(\la')$ of different partitions of~$[n]$ of type~$\la'$. Since $\S_n$ acts transitively on~$\P_n(\la')$, this number is equal to
$\frac{n!}{\#\Stab(\al)}$ where $\Stab(\al)$ is
the cardinality of the stabilizer of an arbitrary partition $\al\in\P_n(\la')$. Let $\la'=(1^{m_1}2^{m_2}\ldots n^{m_n})$. It is easy to see that this stabilizer is isomorphic to the following subgroup of~$\S_n$:
\begin{equation}\label{wreath}
W_{\la'}=(\S_1\wr\S_{m_1})\times(\S_2\wr\S_{m_2})\times\ldots\times(\S_n\wr\S_{m_n}),
\end{equation}
where $\wr$ stands for wreath product.
Then
$$
\#{\cal A}_\la=\frac{n!}{\#W_{\la'}}=\frac{n!}{\prod_k (k!)^{m_k}m_k!}.
$$

\smallskip\noindent{\bf Example~1. The standard representation (braid arrangement).}
Let $\la=(n-1,1)$. Then $\la'=(21^{n-2})$, so a collection of transpositions of type~$\la'$ is just a transposition~$\tau$, the corresponding Young group $\S_\tau$ is the group (isomorphic to~$\S_2$) generated by this transposition, and the corresponding
hyperplane~$H_\tau$, which is the orthogonal complement to the sign representation of~$\S_\tau$, is just the invariant subspace $H_\tau=\{v: \tau v=v\}$ of~$\tau$. Thus, we obtain the braid arrangement. More exactly, the classical braid arrangement~$\Br_n$ is defined by the same construction in the space~$V_{\rm nat}=\C^n$ of the {\it natural} representation $\pi_{\rm nat}=\pi_{(n-1,1)}+\pi_{(n)}$ of $\S_n$, where the one-dimensional subspace corresponding to the identity representation~$\pi_{(n)}$  belongs to all hyperplanes. Thus, $\Br_n=\widetilde{\Br}_n\times\Phi_1$ where $\Phi_1$ is the empty $1$-arrangement and $\widetilde{\Br}_n$ is an irreducible arrangement which exactly coincides with~$\A_{(n-1,1)}$.

\medskip

\noindent{\bf Example~2.} Let $\la=(2,1^{n-2})$. Denote by $\chi_n$ the characteristic polynomial\footnote{Recall that the characteristic polynomial~$\chi(t)$ of a hyperplane arrangement and its Poincar\'e polynomial~$\Poin(t)$ are related by the formula $\chi(t)=t^\ell\Poin(-t^{-1})$ where $\ell$ is the dimension of the ambient space.} of the arrangement~$\A_\la=\A_{(2,1^{n-2})}$ (that is, the characteristic polynomial of the intersection lattice of~$\A_\la$). In this case, we can calculate~$\chi_n$ explicitly.

\begin{proposition} The characteristic polynomial of $\A_\la$ is given by the formula
$$
\chi_n(t)=(t-1)^{n-1}-(t-1)^{n-2}+\ldots+(-1)^n(t-1)=\frac{t-1}t\big((t-1)^n-(-1)^{n-1}\big).
$$
\end{proposition}

\begin{proof}
Since $\la'=(n-1,1)$, the hyperplanes of~$\A_\la$ are indexed by ${(n-1)}$\nobreakdash-sub\-sets of~$[n]$, or, passing to complements, by $j\in[n]$.
Observe that $\pi_{(2,1^{n-2})}|_{\S_{n-1}}=\pi_{(1^{n-1})}+\pi_{(2,1^{n-3})}$,
so the hyperplane~$H_j$ is exactly the space of the representation~$\pi_{(2,1^{n-3})}$ of the Young subgroup~$\S_{\widehat j}$, where we denote $\widehat j=[n]\setminus\{j\}$. It follows that for any $A\subset[n]$ with ${|A|=k\le n-2}$, the intersection $\cap_{j\in A}H_j$ is the $(n-k-1)$-dimensional space of the representation~$\pi_{(2,1^{n-k-2})}$ of the Young subgroup~$\S_{[n]\setminus A}$.

Consider the distinguished hyperplane $H_n$. Let $\A'=\A\setminus\{H_n\}$ and ${\A''=\A^{H_n}}$ be the restriction of $\A$ to $H_n$. Then it is not difficult to see from above that $\A''=\A_{(2,1^{n-3})}$, while $\A'$ is the Boolean arrangement in the $(n-1)$-dimensional space.  Since the characteristic polynomial~$\chi_{\A'}$ of the Boolean arrangement is equal to $(t-1)^{n-1}$, from the deletion-restriction theorem $\chi_{\A}(t)=\chi_{\A'}(t)-\chi_{\A''}(t)$ (see \cite[Corollary~2.57]{OT}) we obtain a recurrence relation $\chi_n(t)=(t-1)^n-\chi_{n-1}(t)$ for~$\chi_n$. Since, obviously, $\chi_2(t)=t-1$, the desired formula follows.
\end{proof}

\smallskip\noindent{\bf Example 3.} The table below shows the basic characteristics of the intrinsic arrangements~$\A_\la$ for all diagrams~$\la$ with $5$ cells (except the trivial cases of the identity and sign representations, where $V_\la$ is one-dimensional): the dimension~$\dim V_\la$ of the ambient space, the number~$\#\A_\la$ of hyperplanes, and the characteristic polynomial~$\chi_{\A_\la}(t)$.

\medskip
   \begin{tabular}{|c|c|c|c|}
      \hline
      $\la$ & $\dim V_\la$ & $\#\A_\la$ & $\chi_{\A_\la}(t)$\\
      \hline
      $(41)$ & 4 & 10& $(t-4)(t-3)(t-2)(t-1)$\\
      $(32)$ & 5 & 15 & $t^5-15t^4+90t^3-260t^2+350t-166$\\
      $(31^2)$ & 6& 10& $t^6-10t^5+45t^4-115t^3+175t^2-147t+51$\\
      $(21^3)$ & 4 & 5 & $t^4-5t^3+10t^2-10t+4$\\
\hline
   \end{tabular}

\section{The ``coordinate'' model of the intrinsic arrangement}\label{sec:tensor}

The intrinsic arrangement ${\cal A}_\la$ is defined for the irreducible representation $\pi_\la$ of $\S_n$. However, this representation has different realizations. By a~slight abuse of language, {\it we will denote by the symbol ${\cal A}_\la$ the corresponding arrangement in whatever space the representation $\pi_\la$ is realized}; this should not cause any ambiguity.

Let  $\la'=(1^{m_1}2^{m_2}\ldots n^{m_n})$, and let $W_{\la'}$ be the group given by~\eqref{wreath}. For each subgroup of the form $\S_k\wr\S_{m_k}$, consider the representation $\sgn_k\wr\id_{m_k}$ obtained by taking the sign representation in each factor $\S_k$ and the identity representation of $\S_{m_k}$. Consider the corresponding induced representation of $\S_n$:
\begin{equation}\label{wreathrep}
\xi_{\la'}=\Ind_{W_{\la'}}^{\S_n}\big((\sgn_1\wr\id_{m_1})\times(\sgn_2\wr\id_{m_2})\times\ldots\times(\sgn_n\wr\id_{m_n})\big).
\end{equation}
Observe that $\dim\xi_{\la'}=\#{\cal A}_\la$.

\begin{lemma}\label{multone}
The multiplicity of $\pi_\la$ in the decomposition of $\xi_{\la'}$ into irreducible representations is equal to~$1$.
\end{lemma}

\begin{proof}
Consider the Young subgroup $\S_{\la'}\subset\S_n$ corresponding to the diagram~$\la'$, and let $\eta_{\la'}=\Ind_{\S_{\la'}}^{\S_n}\sgn_{\la'}$ be the representation of $\S_n$ induced from the sign representation of $\S_{\la'}$. It is well known from the classical representation theory of symmetric groups that the multiplicity of $\pi_\la$ in $\eta_{\la'}$ is equal to~$1$.

Now, observe that we have a natural embedding $\S_{\la'}\subset W_{\la'}$ and denote $\Xi_{\la'}=\Ind_{\S_{\la'}}^{W_{\la'}}\sgn_{\la'}$. By the transitivity of induction, we have $\eta_{\la'}=\Ind_{W_{\la'}}^{\S_n}\Xi_{\la'}$.  It is not difficult to see that
$$
\Xi_{\la'}=(\sgn_1\wr\Reg_{m_1})\times(\sgn_2\wr\Reg_{m_2})\times\ldots\times(\sgn_n\wr\Reg_{m_n}),
$$
where $\Reg_k$ is the regular representation of $\S_k$. Comparing with~\eqref{wreathrep},
it follows that $\xi_{\la'}\subset\eta_{\la'}$, and hence the multiplicity of $\pi_\la$ in the decomposition of $\xi_{\la'}$ is at most~$1$.

It remains to prove that $\pi_\la$ is indeed contained in $\xi_{\la'}$. To this end, recall that the representation $\eta_{\la'}$ can be realized in the space $M^{\la'}$ spanned by the elements $e_t=\sum_{\si\in R_t}\sgn(\si)\si(t)$ where $t$ is a Young tableau of shape~$\la'$ and $R_t$ is the row stabilizer of~$t$ (and we mean the permutation action of~$\S_n$ on Young tableaux). The subrepresentation~$\pi_\la$ is spanned by the elements of the form
\begin{equation}\label{symmetrizer}
\sum_{\si\in C_t}\si(e_t)
\end{equation}
 where $C_t$ is the column stabilizer of~$t$. Now, observe that the subspace of~$M^{\la'}$ corresponding to~$\xi_{\la'}$ consists of the linear combinations of~$e_t$ such that if tableaux $t_1$ and $t_2$ differ only by the order of rows of equal length, then $e_t$ and $e_{t'}$ have equal coefficients. But it is easy to see that the elements~\eqref{symmetrizer} satisfy this property, hence $\pi_\la\subset\xi_{\la'}$.
\end{proof}

\smallskip\noindent{\bf Remark.} If $\la'$ has no rows of equal length greater than~$1$, then the group~$W_{\la'}$ coincides with the Young subgroup~$\S_{\la'}$ and the representation $\xi_{\la'}$ coincides with $\eta_{\la'}$. In this case, Lemma~\ref{multone} reduces to a key fact of the classical approach to the representation theory of symmetric group (the Frobenius--Young correspondence).

\begin{theorem}\label{th:tensor}
Let ${\cal B}_{\la}$ be the Boolean (coordinate) arrangement of hyperplanes in the space~$M^{\la'}$ of $\xi_{\la'}$. Then the restriction of ${\cal B}_{\la}$ to the subspace ${V_\la\subset M^{\la'}}$ of the irreducible representation $\pi_\la$ coincides with the intrinsic arrangement~${\cal A}_\la$.
\end{theorem}

\begin{proof}
By Lemma~\ref{multone}, there is a unique subspace $V_\la\subset M^{\la'}$ such that the restriction of $\xi_{\la'}$ to $V_\la$ is isomorphic to $\pi_\la$. Denote by $P:M^{\la'}\to V_\la$ the orthogonal projection to $V_\la$; it is an operator commuting with the action of~$\S_n$. Now,  consider the distinguished tableau~$\tau^{\min}_\la$ and the conjugate tableau~$\tau^{\max}_{\la'}=(\tau^{\min}_\la)'$ of shape~$\la'$.
Let $w_0\in M^{\la'}$ be the tabloid corresponding to~$\tau^{\max}_{\la'}$. Note that $w_0$ is a cyclic vector of $M^{\la'}$ and, consequently, $Pw_0$ is a cyclic vector of $V_\la$, so $Pw_0\ne0$. Further, $w_0$ is the normal to one of the hyperplanes~$B_0$ of the coordinate arrangement~${\cal B}_{\la}$. Hence $Pw_0$ is the normal to $B_0\cap V_\la$. By construction, the distinguished Young subgroup $\S_{\al_0}$ acts on the vector $Pw_0$ as the sign representation. Since, by Lemma~\ref{l:hyper}, the multiplicity of the sign representation of~$\S_{\al_0}$ in $\pi_\la$ is equal to~$1$, it follows that $Pw_0$ coincides (up to a multiplicative constant) with $n_{\al_0}$ and, therefore, $B_0\cap V_\la=H_{\al_0}$, i.e., $B_0\cap V_\la\in{\cal A}_\la$. The rest follows from the fact that all the other hyperplanes both from ${\cal B}_{\la}$ and ${\cal A}_\la$ are obtained from $B_0$ and $H_{\al_0}$, respectively, by the action of~$\S_n$ and the projection~$P$ commutes with this action.
\end{proof}

\section{Representations corresponding to hook diagrams}\label{sec:hook}

Let $\la=(n-k,1^k)$ be an arbitrary hook diagram. Then $M^{\la'}=\La^{k+1}(\C^n)$ is the $(k+1)$th exterior power of~$\C^n$
(the space
 of totally antisymmetric tensors~$T$ of valence~$k+1$), the coordinate arrangement ${\cal B}_{(n-k,1^k)}$ consists of the ${n\choose{k+1}}$ ``coordinate'' hyperplanes $H_{I}:=\{T_{I}=0\}$,  where by $I=\{i_1,\ldots, i_{k+1}\}$ we denote a ${(k+1)}$\nobreakdash-subset of~$[n]$, assuming that $1\le i_1<i_2<\ldots<i_{k+1}\le n$, and the restriction of this arrangement to~$V_{(n-k,1^k)}$ is exactly the arrangement~${\cal A}_{(n-k,1^k)}$.

The representation of $\S_n$ in $\La^{k+1}(\C^n)$ decomposes into the sum of two irreducible representations isomorphic to $\pi_{(n-k,1^k)}$ and $\pi_{(n-k-1,1^{k+1})}$. Denote the corresponding subspaces by $V^{(k+1)}_{k}$ and $V^{(k+1)}_{k+1}$, respectively, so $\La^{k+1}(\C^n)=V^{(k+1)}_{k}\oplus V^{(k+1)}_{k+1}$.
Then the equivariant embedding $V^{(k)}_{k}\hookrightarrow V^{(k+1)}_{k}$ is given by the formula $V^{(k)}_{k}\ni\al\mapsto\partial\al\in V^{(k+1)}_{k}$ where $\partial:\La^{k}(\C^n)\to\La^{k+1}(\C^n)$ is the usual differential in the exterior algebra:
$$
(\partial T)_{i_1\ldots i_{k+1}}=\sum_{j=1}^{k+1}(-1)^{j+1}T_{i_1\ldots \widehat{i_j}\ldots i_{k+1}}
$$
(as usual, a hat over an index means that the index is omitted). The following theorem shows that the arrangement ${\cal A}_{(n-k,1^k)}$ can be essentially given by the set of equations $\partial T=0$.

\begin{theorem}\label{th:hook}
Consider the hyperplane arrangement ${\cal C}_{(n-k,1^k)}=\{H_{I}\}$ in $\La^{k}(\C^n)$  where
$$
H_{I}=\{T\in\La^{k}(\C^n):(\partial T)_{I}=0\}
$$
and $I$ ranges over all $(k+1)$-subsets of~$[n]$. Then
$$
{\cal C}_{(n-k,1^k)}={\cal A}_{(n-k,1^k)}\times\Phi_\ell
$$
where $\Phi_{\ell}$ is the empty arrangement in the $\ell$-dimensional space with $\ell=\dim\pi_{(n-k+1,1^{k-1})}={n-1\choose k-1}$.
\end{theorem}

\begin{proof}
Follows from the above considerations and the identity $\partial^2=0$. Namely, each hyperplane from ${\cal C}_{(n-k,1^k)}$ is the direct sum of a hyperplane from ${\cal A}_{(n-k,1^k)}$ and the space $V^{(k)}_{k-1}$.
\end{proof}

\medskip
This construction is a direct generalization of the case of the braid arrangement~$\Br_n$ corresponding to the standard representation of~$\S_n$. Namely, let $\la=(n-1,1)$, that is, $k=1$. Then $\La^{2}(\C^n)$ can be identified with the space~$M_n^{\rm skew}=\{(a_{ij})_{i,j}^n:a_{ij}=-a_{ji}\}$ of skew-symmetric $n\times n$ matrices. The subspace $V^{(2)}_{1}$ of the representation~$\pi_{(n-1,1)}$ is given by
$$
V^{(2)}_{1}=\{(a_{ij})_{i,j}^n: a_{ij}=\al_i-\al_j \mbox{ where } (\al_j)_{j=1}^n\in \C^n,\;\sum_{j=1}^n\al_j=0\}.
$$
The coordinate arrangement ${\cal B}_{(n-1,1)}$ in $M_n^{\rm skew}$ consists of the $\frac{n(n-1)}2$ hyperplanes $H_{ij}:=\{M\in M_n^{\rm skew}: a_{ij}=0\}$ for $1\le i<j\le n$, and the restriction of this arrangement to~$V^{(2)}_{1}$ is exactly the irreducible braid arrangement~$\widetilde{\Br}_n$. While the arrangement~${\cal C}_{(n-1,1)}$ in $\C^n$ consists of the hyperplanes of the form $\{\al_i-\al_j =0\}$, which is the standard definition of the braid arrangement~$\Br_n$. Thus, ${\cal C}_{(n-1,1)}=\Br_n$ and ${\cal B}_{(n-1,1)}=\widetilde{\Br_n}$.

Arnold's result  \cite{Arnold}  that the complement of
the braid arrangement is a~$K(\pi,1)$ space  ignited attempts to generalize it to a wider class
of reflection arrangements.
This generalization was proved in full by Bessis~\cite{Bessis}.
In the remaining part of this section, we will prove that the complements of our arrangements corresponding to hook diagrams with $k>1$ are never $K(\pi,1)$. In particular, in this case the group generated by all reflections
at the hyperplanes is infinite.

Suppose there exists a linear
dependence between the left-hand sides of several hyperplane equations, i.e., there exist nonzero
$a_1,\ldots, a_m\in\C$  and ${(k+1)}$-subsets $I_1,\ldots, I_m\subset[n]$
 such that
$\sum_{r=1}^ma_r(\partial T)_{I_r}=0$.
Here the vanishing of the left-hand side means that  the coefficients of the elements
with the same index sum to~$0$. Thus we can focus our attention on the indices $I_r$ only, which brings
an abstract simplicial complex into the picture.

The abstract complex we need is just the simplex $S$ on the set of vertices~$[n]$ (i.e., each
 subset of $[n]$ is a simplex of~$S$).
Consider the chain complex $C=C(S)$ on $S$ with coefficients from $\C$ and denote its
boundary map by~$d$. Then the above equality is equivalent to
$\sum_{r=1}^ma_rdI_r=d(\sum_{r=1}^ma_rI_r)=0$,
where $I_r$ is viewed as a simplex in the abstract complex and as a generator of the
respective chain complex. Thus,
$z=\sum_{r=1}^ma_rI_r$ is a cycle in the chain complex.
The following observation is obvious.

\begin{lemma}\label{l:cycle}
If a chain in a simplicial complex is a cycle of dimension greater than one,
then it cannot be a linear combination of less than $4$ natural generators (simplices).
\end{lemma}

Lemma~\ref{l:cycle} immediately implies the following.

\begin{lemma}\label{l:double}
The intersection of any two hyperplanes in
$\A_{(n-1,1^k)}$ for ${k>1}$
is double (i.e., there is no other hyperplane containing this intersection).
\end{lemma}

\begin{theorem}\label{th:hookcompl}
Let $\lambda=(n-k,1^k)$ with $k>1$. Then the fundamental group~$\pi_1$ of the
complement to the
arrangement $\A_{\lambda}$ is free Abelian of rank equal to the size of the arrangement.
 Besides, the complement is not $K(\pi,1)$.
 \end{theorem}
\begin{proof}
By factoring out we can reduce the problem to an essential arrangement.

Recall from~\cite{OT} that the generators of $\pi_1$ are one for each hyperplane, and relations correspond to every intersection
of hyperplanes of codimension~$2$. More precisely, assume that $m$ hyperplanes have a common subspace~$J$ of codimension~$2$ and there are no other hyperplanes containing~$J$. Denote the generators corresponding to these $m$
hyperplanes by $a_1,a_2,\ldots,a_m$. Then the relations corresponding to~$J$ are
$$a_1a_2\cdots a_m=a_2a_3\cdots a_ma_1= \cdots=a_ma_1\cdots a_{m-1}.$$
As Lemma~\ref{l:double} implies, in our case for each codimension~$2$ subspace~$J$ we have $m\leq2$, and $m=2$ only for  $J$ that lies in two
hyperplanes. Thus, we have the relations
$ab=ba$ for any two generators $a$ and $b$. The group given by these generators and relations is the free Abelian group
of rank equal to the number of hyperplanes.

Now, if the complement were $K(\pi,1)$, then it would have been of homotopy type of a torus of dimension
$N={n \choose k+1}$. In particular, the homology of the complement must be nonzero in dimension $N$.
 On the other hand, $N>n$, therefore, the rank of the intersection lattice is less than $N$. Thus the homology of the complement
 is 0 in dimension $N$, whence its homotopy type is not that of a torus.
 \end{proof}

\smallskip\noindent {\bf Example.}
 Let $n=5$, $k=2$, and $\lambda=(3,1^2)$. There are $10={5 \choose 3}$
  hyperplanes, given in the space $\Lambda^2(\C^5)$ of dimension~$10$
  with  coordinates $t_{ij}$ by the
 following equations: for each triple $s=ijk$ ($i<j<k$), the equation $E(s)$ is
 $$t_{ij}-t_{ik}+t_{jk}=0.$$
 All the hyperplanes have the common subspace $V_{\mu}$ of dimension~$4$ where $\mu=(2,1^3)$. Factoring out the common subspace, we
  get an essential arrangement in $V_{\lambda}$ of dimension~$6$.
  The left-hand sides of the $10$ equations above have $5$ linear dependencies
corresponding to the standard basis $T_{ijkl}$ ($i<j<k<l$) of
$\Lambda^4(\C^5)$. For instance,
the relation corresponding to the basic element $T_{1234}$ is
$E(123)-E(124)+E(134)-E(234)=0.$

\section{Representations of the partition lattice}\label{sec:partlat}

Let $\Pi_n=(\P_n,\prec)$ be the partition lattice, i.e., the set of partitions of the set~$[n]$ ordered by refinement.  It is easy to see that
the set $\YS_n$ of all Young subgroups in~$\S_n$ ordered by inclusion is a lattice, and the map
$$
\psi:\Pi_n\to\YS_n,\qquad \psi(\al)=\S_\al,
$$
is a lattice isomorphism between $\YS_n$ and the partition lattice $\Pi_n$.

It is well known that the set of all subspaces of a vector space ordered by inclusion is also a lattice, where the greatest lower bound and the least upper bound are given by the intersection and the sum of subspaces, respectively.
Now, for a fixed Young diagram~$\lambda$ of size~$n$, denote by ${\cal S}(V_\la)$ the lattice of subspaces of~$V_\la$. We have a map $i_\la$ that associates with a subgroup $G\subset\S_n$ its invariant subspace
$V_\la^{G}$ in $V_\la$. Considering the composite $\phi_\la=i_\la\circ\psi$, we obtain a map
$$
\phi_\la:\Pi_n\to{\cal S}(V_\la)
$$
that sends a partition $\al$ of $[n]$ to the invariant subspace $V_\la^{\S_\al}$.

Given $\al\in\Pi_n$, denote by $\bar\al\in\Y_n$ the partition of the integer~$n$ determined by the sizes of blocks of $\al$.

\begin{proposition}\label{prop:prop} Fix $\la\in\Y_n$ and denote $\phi:=\phi_\la$. Then
\begin{enumerate}
\itemsep=-2mm
\item[\rm (i)] $\phi(\al\vee\beta)=\phi(\al)\wedge\phi(\beta)$;
\item[\rm (ii)] $\phi(\al\wedge\beta)\supset\phi(\al)\vee\phi(\beta)$;
\item[\rm (iii)] $\dim\phi(\al)=K_{\la\bar\al}$;
\item[\rm (iv)] $\phi(\al)=0$ unless $\la\trianglerighteq\bar\al$ (where $\trianglerighteq$ stands for the dominance order on partitions).
\end{enumerate}
\end{proposition}
\begin{proof}
(i) Denote $V=V_\la$. We have $\phi(\al)\wedge\phi(\beta)=\phi(\al)\cap\phi(\beta)= {V^{\S_\al}\cap V^{\S_\beta}}=V^{\langle \S_\al,\S_\beta\rangle}=V^{\psi(\al\vee\beta)}=\phi(\al\vee\beta)$, where by $\langle \S_\al,\S_\beta\rangle$ we denote the subgroup in $\S_n$ generated by $\S_\al$ and $\S_\beta$, which is exactly the Young subgroup corresponding to $\al\vee\beta$.

(ii) We have $\phi(\al)\vee\phi(\beta)=V^{\S_\al}+V^{\S_\beta}\subset V^{\S_\al\cap\S_\beta}=V^{\S_{\al\wedge\beta}}=\phi(\al\wedge\beta)$.

(iii) Follows from Lemma~\ref{l:kostka}.

(iv) Follows from~(iii) and the upper triangularity of Kostka numbers (see, e.g.,~\cite[Section~I.6]{Mac}).
\end{proof}

\smallskip\noindent{\bf Remark.} As one can easily see, the equality in~(ii) does not hold in general.

\begin{corollary}
The Kostka numbers satisfy the following inequality: for every $\la\in\Y_n$ and any $\al,\beta\in\Pi_n$,
$$
K_{\la,\overline{\al\vee\beta}}+K_{\la,\overline{\al\wedge\beta}}\ge K_{\la,\bar\al}+K_{\la,\bar\beta}.
$$
\end{corollary}

\begin{proof}
By Proposition~\ref{prop:prop}, the right-hand side is equal to $$\dim\phi(\al)+\dim\phi(\beta)=\dim(\phi(\al)\wedge\phi(\beta))+\dim(\phi(\al)\vee\phi(\beta)).$$ But $\dim(\phi(\al)\wedge\phi(\beta))=\dim\phi(\al\vee\beta)=K_{\la,\overline{\al\vee\beta}}$, while $$\dim(\phi(\al)\vee\phi(\beta))\le\dim\phi(\al\wedge\beta)=K_{\la,\overline{\al\wedge\beta}}.$$
\end{proof}

Thus, in the lattice ${\cal S}(V_\la)$ of all subspaces of~$V_\la$ we have the subset $\SY(V_\la)=\phi(\Pi_n)$ of subspaces invariant with respect to Young subgroups of~$\S_n$. It follows from Proposition~\ref{prop:prop} that the intersection $V_1\wedge V_2$ of two subspaces $V_1,V_2\in\SY(V_\la)$ also lies in $\SY(V_\la)$. However, their sum $V_1\vee V_2$ does not necessarily lie in $\SY(V_\la)$. Nevertheless, the following statement holds.

\begin{lemma}\label{l:lat}
The set $\SY(V_\la)$ of subspaces in $V_\la$ invariant with respect to Young subgroups of~$\S_n$ ordered by inclusion is a coatomistic lattice.
\end{lemma}

\begin{proof}
As we have observed, the intersection of two subspaces from $\SY(V_\la)$ belongs to $\SY(V_\la)$, which implies that $\SY(V_\la)$ is a meet-semilattice. Besides, it contains the greatest element $\hat 1=V_\la=\phi(\eps)$ where $\eps$ is the partition into separate points: $\eps=\{\{1\},\{2\},\ldots,\{n\}\}$. But it is well known (see, e.g., \cite[Proposition~3.3.1]{St}) that a meet-semilattice with $\hat 1$ is a lattice.

Obviously, for every $\al\in\Pi_n$ with $\bar\al=(2,1^{n-2})$, the corresponding subspace $V_\la^{\S_\al}$ is a coatom of $\SY(V_\la)$. The fact that every element $V_\la^\beta$ of $\SY(V_\la)$ is the greatest lower bound of atoms follows from the considerations of Section~\ref{sec:first}: it suffices to take a collection of transpositions that generate $\S_\beta$; each transposition $(ij)$ gives rise to a partition $\al_{ij}\in\Pi_n$ with $\bar\al_{ij}=(2,1^{n-2})$ where $\al_{ij}$ consists of the 2-block $\{ij\}$ and $n-2$ singletons, and $V_\la^\beta$ is the greatest lower bound of the corresponding coatoms of~$\SY(V_\la)$.
\end{proof}

We emphasize that the meet of two elements in $\SY(V_\la)$ coincides with their meet in ${\cal S}(V_\la)$, but for the join this is, in general, not the case. So, $\SY(V_\la)$ is {\it not} a sublattice of ${\cal S}(V_\la)$.

Also, observe that, as we have mentioned in the proof of Lemma~\ref{l:lat}, the greatest element of $\SY(V_\la)$ is $\hat 1=V_\la$, while the smallest element $\hat0$ is the zero subspace $\{0\}$ for all $\la\ne(n)$.

Summarizing, we obtain the following.

\begin{theorem}
Denote by $\SY(V_\la)^*$ the order dual of $\SY(V_\la)$ (which is an atomistic lattice). Then the map $\phi:\Pi_n\to \SY(V_\la)^*$ is a join homomorphism of lattices.
\end{theorem}

If $\pi_{\rm nat}=\pi_{(n-1,1)}+\pi_{(n)}$ is the natural representation of~$\S_n$ in $\C^n$, then $\SY(V_{\rm nat})^*$ is exactly the intersection lattice of the braid arrangement $\Br_n$, which is well known to be isomorphic to the partition lattice~$\Pi_n$ (see, e.g,~\cite{OT}); thus, in this case $\phi=\phi_{\rm nat}$ is in fact an isomorphism of lattices.

\begin{theorem}\label{th:hyper}
The lattice $\SY(V_\la)^*$ is embedded into the intersection lattice $L(\A_\la)$ of the hyperplane arrangement~$\A_\la$, which is a minimal hyperplane arrangement satisfying this property.
\end{theorem}

\begin{proof}
Recall that $\SY(V_\la)^*$ is an atomistic lattice, the atoms being  the subspaces~$V^\si_\la$ for all transpositions $\si\in\S_n$. Thus, to show that ${\SY(V_\la)^*\subset L(\A_\la)}$, it suffices to prove that every such atom can be obtained as the intersection of a collection of hyperplanes~$H_\al$. For $i,j\in[n]$ and $\al\in\P_n(\la')$, we write $i\sim_\al j$ if $i$ and $j$ lie in the same block of~$\al$.
Let us show that for every transposition~$\si=(ij)$,
\begin{equation}\label{atoms}
V^\si_\la=\bigcap_{\al\in\P_n(\la'):\;i\sim_\al j}H_\al.
\end{equation}
Obviously (since everything is invariant under~$\S_n$), it suffices to prove this for $\si=(12)$. Let $\al_0$ be the distinguished partition (see Sec.~\ref{sec:first}).
 Any other partition~$\al$ of type~$\la'$ such that $1\sim_\al 2$ is obtained from $\al_0$ by conjugation by some element $g\in\S_{\{3,\ldots,n\}}$, and then $H_\al=gH_{\al_0}$. It follows that the orthogonal complement to the right-hand side of~\eqref{atoms} is the subspace spanned by $\S_{\{3,\ldots,n\}}n_{\al_0}$. But this is exactly the subspace in~$V_\la$ spanned by the Gelfand--Tsetlin vectors $e_t$ indexed by Young tableaux~$t$ such that $2$ lies in the first column of~$t$, while $V^\si_\la$ is the subspace in~$V_\la$ spanned by the vectors $e_t$ indexed by Young tableaux~$t$ such that $2$ lies in the first row of~$t$. Thus, these are the orthogonal complements to each other, and the first assertion of the theorem follows.

 On the other hand, each hyperplane $H_\al$ of~$\A_\la$ is, by construction, obtained as the join of elements of~$\SY(V_\la)^*$, which implies the minimality.
\end{proof}

\end{document}